\documentclass[12pt,english,refpage, intoc, refeq]{article}
\usepackage[T1]{fontenc}
\usepackage[latin9]{inputenc}
\usepackage{xcolor,verbatim,amsthm,amsmath,amssymb}
\usepackage{geometry}
\makeatletter

\numberwithin{equation}{section}
\numberwithin{figure}{section}
\theoremstyle{plain}
\newtheorem{thm}{\protect\theoremname}
  \theoremstyle{definition}
  \newtheorem{example}[thm]{\protect\examplename}
  \theoremstyle{definition}
  \newtheorem{defn}[thm]{\protect\definitionname}
  \theoremstyle{plain}
  \newtheorem{lem}[thm]{\protect\lemmaname}
  \theoremstyle{plain}
  \newtheorem{prop}[thm]{\protect\propositionname}

\newcommand{\dint}{\displaystyle\int}
\usepackage{datetime,color,currfile,amssymb,amsmath,pdfsync,hyperref}

\date{ }
\makeatother

\usepackage{babel}
  \providecommand{\definitionname}{Definition}
  \providecommand{\examplename}{Example}
  \providecommand{\lemmaname}{Lemma}
  \providecommand{\propositionname}{Proposition}
\providecommand{\theoremname}{Theorem}

\begin{document}

\title{Mixed stochastic differential equations: Existence and uniqueness result}

\author{\textbf{Jos{\'e} Lu{\'\i}s da Silva}\\
CCM, University of Madeira, Campus da Penteada,\\
9020-105 Funchal, Portugal.\\
Email: luis@uma.pt 
\and 
\textbf{Mohamed Erraoui}\\
Universit{\'e} Cadi Ayyad, Facult{\'e} des Sciences Semlalia,\\
 D{\'e}partement de Math{\'e}matiques, B.P. 2390, Marrakech, Maroc\\
Email: erraoui@uca.ma 
\and 
\textbf{El Hassan Essaky}\\
Universit{\'e} Cadi Ayyad, Facult{\'e} Poly-disciplinaire\\
 Laboratoire de Mod{\'e}lisation et Combinatoire\\
D{\'e}partement de Math{\'e}matiques et d'Informatique B.P. 4162,\\
Safi, Maroc.\\  Email: essaky@uca.ma}

\maketitle
\begin{abstract}
In this paper we shall establish an existence and uniqueness result
for solutions of multidimensional, time dependent, stochastic differential
equations driven simultaneously by a multidimensional fractional Brownian
motion with Hurst parameter $H>1/2$ and a multidimensional standard
Brownian motion under a weaker condition than the Lipschitz one. \medskip{}

\end{abstract}

\section{Introduction}
The fractional Brownian  motion (fBm for short) $B^{H}=\{B^{H}(t) , t\in [0,T]\}$ with Hurst  parameter $H\in (0,1)$  is a Gaussian self-similar process with stationary increments.
This process was introduced by Kolmogorov \cite{kol} and studied by Mandelbrot and Van Ness in \cite{MN}, where a stochastic integral representation in terms of a standard
Brownian motion (Bm for short) was established. The parameter
$H$ is called Hurst index from the
statistical analysis, developed by the climatologist Hurst \cite{hurst}. The self-similarity and stationary increments properties make the fBm an appropriate model for many applications in diverse fields from biology to finance. From the properties of the fBm it follows that, for every $\alpha >0$
$$
\mathbb{E}\left(|B^H(t)-B^H(s)|^{\alpha}\right) = \mathbb{E}\left(|B^H(1)|^{\alpha}\right)|t-s|^{\alpha H}.
$$
As a consequence of the Kolmogorov continuity theorem, we deduce that there exists a version of the fBm $B^H$ which is a continuous process and whose paths are $\gamma$-H\"{o}lder continuous for every $\gamma <H$. Therefore, the fBm with Hurst parameter $H\neq \frac12$ is not a semimartingale and then the It\^{o} approach to the construction of stochastic integrals with respect to fBm is not valid. Two main approaches have been used in the literature to define stochastic integrals with respect to fBm with Hurst parameter $H$. Pathwise  Riemann-Stieltjes  stochastic integrals can be defined using Young's integral \cite{young} in the case $H>\frac 12$. When $H\in (\frac14, \frac12)$, the rough path analysis introduced by Lyons \cite{lyons} is a suitable method to construct pathwise stochastic integrals.

A second approach to develop a stochastic calculus with respect to the fBm is based on the techniques of Malliavin calculus. The divergence operator, which is the adjoint of the derivative operator, can be regarded as a stochastic integral, which coincides with the limit of Riemann sums constructed using the Wick product.
 This idea has been developed by
Decreusefond and \"{U}st\"{u}nel \cite{DU}, Carmona, Coutin and Montseny \cite{CC}, Al{\`o}s, Mazet and Nualart \cite{AMN1, AMN2}, Al{\`o}s and Nualart \cite{AN} and Hu \cite{hu}, among others.  The integral constructed by this method has zero mean.

Let $T>0$ be a fixed time and $\big(\Omega,\mathcal{F},(\mbox{\ensuremath{\mathcal{F}}}_{t})_{t\in[0,T]},P\big)$
be a given filtered complete probability space with $\left(\mbox{\ensuremath{\mathcal{F}}}_{t}\right)_{t\in[0,T]}$
being a filtration that satisfies the usual hypotheses. 
The aim of this paper is to study
the following stochastic differential equation (SDE for short) on $\mathbb{R}^{n}$
\begin{equation}
X(t)=x_{0}+\int_{0}^{t}b(s,X(s))\, ds+\int_{0}^{t}\sigma_{W}(s,X(s))dW(s)+\int_{0}^{t}\sigma_{H}(s,X(s))dB^{H}(s),\label{eq:1-1}
\end{equation}
where $t\in\left[0,T\right]$, $x_{0}\in\mathbb{R}^{n},$ $W$ is
a $m$-dimensional standard $\mbox{\ensuremath{\mathcal{F}}}_{t}$-Bm
 and $B^{H}$ a $d$-dimensional $\mbox{\ensuremath{\mathcal{F}}}_{t}$-adapted
fBm. The main difficulty when considering Equation~(\ref{eq:1-1}) lies in the fact that both stochastic integrals are dealt in different ways. However, the integral with respect to the Bm is an Itô integral, while
the integral with respect to the fBm has to be understood in the pathwise sense. Mixing the two integrals makes things difficult, forcing to consider very smooth coefficients to prove existence and uniqueness of solution to Equation~(\ref{eq:1-1}). 

It is well known that, under suitable assumptions on the coefficients
$b,\sigma_{W},\sigma_{H}$ (see below), the Equation~(\ref{eq:1-1})
has a unique solution which is $(H-\varepsilon)$-H{\"o}lder continuous,
for all $\varepsilon>0$. This result was first considered in \cite{kubilius}, where unique solvability was proved for time-independent
coefficients and zero drift. Later, in \cite{zahle}, existence
of solution to (\ref{eq:1-1}) was proved under less restrictive assumptions,
but only locally, i.e.~up to a random time. In \cite{GN}, global existence and uniqueness of solution to the Equation~(\ref{eq:1-1})
was established under the assumption that $W$ and $B^{H}$ are independent.
The latter result was obtained in \cite{MS, MS1} without the independence
assumption. We stress on the fact that all these works consider the Lipschitz case. It should be noted,
in addition, that the Lipschitz condition is the most used to
establish the pathwise uniqueness for ordinary and SDEs via the Gronwall lemma. 
Thus, the following question appears naturally: are there any weaker conditions than the
Lipschitz continuity under which the SDE (\ref{eq:1-1}) has a unique strong solution?

In order to answer the above question our approach is to prove that the Euler's polygonal approximations converge uniformly in $t\in [0,T]$, in probability, to a process, which we show to be the strong solution. The basic tools are the pathwise uniqueness for the SDE (\ref{eq:1-1}), tightness of the sequence of the laws of Euler's approximations and the Skorokhod's embedding theorem. It is important to note that the linear growth condition and the continuity of the coefficients are sufficient for the convergence of the Stieltjes and It{\^o} integrals. However, the integral with respect the fBm needs more regularity. To prove the convergence in probability we use an elementary result due to Gyongy and Krylov \cite{GK} which highlights the famous result of Yamada and Watanabe saying that pathwise uniqueness implies uniqueness in law. It is worth mentioning that the pathwise uniqueness property for the SDE (\ref{eq:1-1}) is obtained under weak assumption than the Lipschitz condition. More precisely our conditions are based on the modulus of continuity of the coefficients that achieve pathwise
uniqueness using Bihari's type lemma. 
It should be noted that such conditions are considered by many authors for the existence and uniqueness of solutions of different kind of equations where the Bihari's lemma is the cornerstone in the proof of these results.

The article is organized as follows. In Section~\ref{preliminaries}, we state our assumptions  on the coefficients $b$, $\sigma_W$  and $\sigma_H$ of Equation~\eqref{eq:1-1}, recall briefly the deterministic fractional calculus in order to define the integral with respect to fBm and introduce proper normed spaces. In addition, we give the definition of strong,  weak solution and pathwise uniqueness of Equation~\eqref{eq:1-1}. In Section~\ref{sec:pathunique}, the pathwise uniqueness property for the  solutions of Equation~\eqref{eq:1-1} is proved (see Theorem~\ref{pathunique} below). Finally, in Section~\ref{sec:Euler}, we define the Euler approximations sequence and prove that it is tight.  Moreover, we show that these approximations converge in probability to a process which turns out to be a strong solution of the SDE~\eqref{eq:1-1}, cf.~Theorem~\ref{thm:Eulerstrong} below. In the Appendix, we recall some technical results which play a great
role in this work.  We also show a version of Bihari's lemma which will be used in the proof of pathwise uniqueness to  SDE (\ref{eq:1-1}).

\section{Preliminaries}\label{preliminaries}

Throughout this paper we assume that the coefficients $b, \sigma_W$ and $\sigma_H$, which are continuous, satisfy, for all $x,y\in\mathbb{R}^{n}$ and $t\in\left[0,T\right]$, the
following hypotheses $\mathbf{(H.1)}$  and $\mathbf{(H.2)}$:

\begin{description}
\item [{Hypothesis~$\mathbf{(H.1)}$.}]  The functions $b$ and $\sigma_W$ have a linear growth and  satisfy suitable modulus of continuity with respect to the variable $x$ uniformly in $t$. 
\end{description}
Hypothesis $\mathbf{(H.1)}$ means that $b$ and $\sigma_W$ satisfy
\begin{eqnarray*}
({\mathbf{H.1.1}}) &  & |b(t,x)|\leq K(1+|x|),\\
({\mathbf{H.1.2}}) &  & |b(t,x)-b(t,y)|^{2}\leq\varrho\big(|x-y|^{2}\big)\\
({\mathbf{H.1.3}}) &  & |\sigma_{W}(t,x)|\leq K(1+|x|),\\
({\mathbf{H.1.4}}) &  & |\sigma_{W}(t,x)-\sigma_{W}(t,y)|^{2}\leq\varrho\big(|x-y|^{2}\big),
\end{eqnarray*}
where
$\varrho$ is a concave increasing function from $\mathbb{R}_{+}$
to $\mathbb{R}_{+}$ such that $\varrho(0)=0$, $\varrho(u)>0$ for
$u>0$ and for some $q>1$ we have 
\begin{equation}
\int_{0^{+}}\dfrac{du}{\varrho^{q}(u^{1/q})}=\infty.\label{eq:rho est}
\end{equation}

\begin{description}
\item [{Hypothesis~$\mathbf{(H.2)}$.}]   The function $\sigma_H$ is continuously differentiable in the second variable $x$. Its derivative, with respect to $x$, is bounded, Lipschitz with respect to the same variable uniformly with respect to the first variable $t$. Moreover, both $\sigma_H$ and its derivative are $\beta$-H\"{o}lder with respect to the first variable $t$  uniformly with respect to the second variable.
 \end{description}
 Hypothesis $\mathbf{(H.2)}$ means that $\sigma_H$ and its derivative satisfy
\begin{eqnarray*}
({\mathbf{H.2.1}}) &  & \left|\partial_{x_{i}}\sigma_{H}(t,x)\right|\leq K \\
({\mathbf{H.2.2}}) &  & \left|\partial_{x_{i}}\sigma_{H}(t,x)-\partial_{x_{i}}\sigma_{H}(t,y)\right|\leq K\left|x-y\right| \\
({\mathbf{H.2.3}}) &  & \left|\sigma_{H}(t,x)-\sigma_{H}(s,x)\right|+\left|\partial_{x_{i}}\sigma_{H}(t,x)-\partial_{x_{i}}\sigma_{H}(s,x)\right|\leq K\left|s-t\right|^{\beta}.
\end{eqnarray*}
\begin{example}
Let us give two examples of such function $\varrho$. Let $q>1$ and  $\delta$
be sufficiently small. Define
\begin{eqnarray*}
\varrho_{1}(u)&:=&\left\{ \begin{array}{ll}
u\log^{1/q}(u^{-1}), & \,\,\,0\leq u\leq\delta \\ \\
\delta\log^{1/q}(\delta^{-1})+\varrho'_1(\delta_{-})(u-\delta), & \,\,\, u>\delta.
\end{array}\right. \\ \\ \\
\varrho_{2}(u)&:=&\left\{ \begin{array}{ll}
u\log^{1/q}(u^{-1})\log^{1/q}\left(\log(u^{-1})\right), & \,\,\,0\leq u\leq\delta\\
\\
\delta\log^{1/q}(\delta^{-1})\log^{1/q}\left(\log(\delta^{-1})\right)+
\varrho'_2(\delta_{-})(u-\delta), & \,\,\, u>\delta.
\end{array}\right.
\end{eqnarray*}
It is easy to see that, for $i=1,2$, the function $\varrho_{i}$ is
concave nondecreasing function satisfying (\ref{eq:rho est}).
\end{example}
We begin by a brief review of the deterministic fractional calculus.
We start with the definition of the integral with respect to fBm as a generalized
Lebesgue-Stieltjes integral, following the work of Z{\"a}hle \cite{zahle}. We fix $\alpha\in(0,1)$.
The Weyl-Marchaud derivatives of $f:[a,b]\longrightarrow\mathbb{R}^{n}$ are given
by:

\[
D_{a+}^{\alpha}f(x)=\dfrac{1}{\Gamma(1-\alpha)}\left(\dfrac{f(x)}{\left(x-a\right)^{\alpha}}+\alpha\int_{a}^{x}\dfrac{f(x)-f(y)}{\left(x-y\right)^{\alpha+1}}\, dy\right)
1\!\!1_{\left(a,b\right)}(x)
\]
and 
\[
D_{b-}^{\alpha}f(x)=\dfrac{\left(-1\right)^{\alpha}}{\Gamma(1-\alpha)}\left(\dfrac{f(x)}{\left(b-x\right)^{\alpha}}+\alpha\int_{x}^{b}\dfrac{f(x)-f(y)}{\left(y-x\right)^{\alpha+1}}\, dy\right)1\!\!1_{\left(a,b\right)}(x),
\]
where $\Gamma(\alpha) =\int_0^{\infty} t^{\alpha -1} e^{-t}dt$ is the Gamma function.
Assuming that $D_{a+}^{\alpha}f_{a+}\in L^{1}[a,b]$ and $D_{b-}^{1-\alpha}g_{b-}\in L^{\infty}[a,b]$,
where $g_{b-}(x)=g(x)-g(b-)$, the generalized (fractional) Lebesgue-Stieltjes
integral of $f$ with respect to $g$ is defined as
\begin{equation}
\int_{a}^{b}f\, dg:=(-1)^{\alpha}\int_{a}^{b}\, D_{a+}^{\alpha}f(x)\, D_{b-}^{1-\alpha}g_{b-}(x)\, dx.\label{eq:frac int}
\end{equation}
If $a\leq c<d\leq b$ then we have 
\[
\int_{c}^{d}f\, dg=\int_{a}^{b}1\!\!1_{(c,d)}f\, dg.
\]
It follows from the H{\"o}lder continuity of $B^{H}$ that $D_{b-}^{1-\alpha}B_{b-}^{H}\in L^{\infty}[a,b]$
almost surely (a.s.~for short). Then, for a function $f$ with $D_{a+}^{\alpha}f\in L^{1}[a,b]$,
we can define the integral with respect to $B^{H}$ through (\ref{eq:frac int}). 

Let $0<\alpha<1/2$ and $\mu\in(0,1]$. We will consider the following normed spaces:
\begin{enumerate}
\item $C^{\mu}$ is the space of $\mu$-H{\"o}lder continuous functions
$f:[0,T]\rightarrow\mathbb{R}^{d}$, equipped with the norm
\[
\|f\|_{\mu}:=\|f\|_{\infty}+\underset{0\leq s<t\leq T}{\sup}\dfrac{\left|f(t)-f(s)\right|}{\left(t-s\right)^{\mu}}<\infty,
\]
where
\[
\|f\|_{\infty}:=\underset{0\leq t\leq T}{\sup}\left|f(t)\right|.
\]
\item $C_{0}^{\mu}$ denotes the space of $\mu$-H{\"o}lder continuous functions $f:[0,T]\longrightarrow \mathbb{R}^d$
such that 
\[
\lim_{\varepsilon\rightarrow0}\left(\sup_{0<|t-s|<\varepsilon}\frac{|f(t)-f(s)|}{(t-s)^{\mu}}\right)=0.
\]
We note that $C_{0}^{\mu}$ is complete and separable with respect to the norm $\| \cdot \|_{\mu}$.
\item $W_{0}^{\alpha,\infty}$ is the space of measurable functions $f:[0,T]\longrightarrow\mathbb{R}^{d}$
such that
\[
\|f\|_{\alpha,\infty}:=\underset{0\leq t\leq T}{\sup}\|f\|_{\alpha,t}<\infty,
\]
where
\[
\|f\|_{\alpha,t}:=\left|f(t)\right|+\int_{0}^{t}\dfrac{\left|f(t)-f(s)\right|}{\left(t-s\right)^{\alpha+1}}\, ds.
\]
\item Finally, $W_{T}^{1-\alpha,\infty}$ denotes the space of measurable
functions $f:[0,T]\longrightarrow\mathbb{R}^{m}$ such that
\[
\|f\|_{1-\alpha,\infty,T}:=\underset{0\leq t\leq T}{\sup}\|f\|_{1-\alpha,\infty,t}<\infty,
\]
where 
\[
\| f\|_{1-\alpha,\infty,t}:=\sup_{0\leq u<v<t}\left(\frac{\left|f(v)-f(u)\right|}{\left(v-u\right)^{1-\alpha}}+\int_{u}^{v}\dfrac{\left|f(y)-f(u)\right|}{\left(y-u\right)^{2-\alpha}}dy\right).
\]
\end{enumerate}
Hence, it is clear that 
\[
\underset{0\leq u<v<t}{\sup}\left|D_{v-}^{1-\alpha}B_{v-}^{H}(u)\right|\leq\dfrac{1}{\Gamma(\alpha)}\| B^{H} \|_{1-\alpha,\infty,t}<\infty,
\]
where the last inequality is a consequence of that fact that the random variable $\| B^{H}\|_{1-\alpha,\infty,t}$
has moments of all orders, see Lemma 7.5 in Nualart  and Rascanu \cite{NR}.
Thus, the stochastic integral with respect to the fBm admits the following estimate
\begin{equation}
\left|\int_{0}^{t}f(s)\, dB^{H}(s)\right|
\leq \frac{1}{\Gamma(\alpha)}\| B^{H}\|_{1-\alpha,\infty,t}\| f\|_{\alpha,1,t}, \label{frac int est}
\end{equation}

where 
\[
\| f \|_{\alpha,1,t}:=\int_{0}^{t}\dfrac{\left|f(s)\right|}{s^{\alpha}}\,ds+\int_{0}^{t}\int_{0}^{s}\dfrac{\left|f(s)-f(y)\right|}{\left(s-y\right)^{\alpha+1}}\, dy\,ds.
\]
We give the definition of strong and weak solution as well as pathwise uniqueness for Equation~\eqref{eq:1-1}.
\begin{defn}[Strong solution]
By a strong solution of Equation (\ref{eq:1-1}) we mean an $\mathcal{F}_{t}$-adapted
continuous process $X(t),t\in[0,T]$ such that there exists an increasing
sequence of stopping times $(T_{R})_{R>0}$ satisfying $\lim_{R\rightarrow\infty}T_{R}=T$
a.s.~and for any $R>0$, we have 
\begin{enumerate}
\item $
\sup_{t\in[0,T]}\mathbb{E}\left[\|X(t\wedge T_{R})\|_{\alpha,t}^{2}\right]<\infty.\label{eq:C1}
$
\item The equation
\begin{eqnarray}
X(t\wedge T_{R}) &=& x_{0}+\int_{0}^{t\wedge T_{R}}b\big(s,X(s)\big)\, ds+\int_{0}^{t\wedge T_{R}}\sigma_{W}\big(s,X(s)\big)\, dW(s)\nonumber \\
& & +\int_{0}^{t\wedge T_{R}}\sigma_{H}\big(s,X(s)\big)\, dB^{H}(s),\label{eq:C2}
\end{eqnarray}
\end{enumerate}
holds a.s.. 
\end{defn}
\begin{defn}[Weak solution]
By a weak solution of Equation (\ref{eq:1-1}) we mean a triplet $(X, W, B^H)$, $\big(\Omega,\mathcal{F}, P\big)$ and $(\mathcal{F}_{t})_{t\in[0,T]}$, such that
\begin{enumerate}
\item $\big(\Omega,\mathcal{F}, P\big)$ is a probability space, and $(\mathcal{F}_{t})_{t\in[0,T]}$ is a filtration, of sub-$\sigma$-algebra of $\mathcal{F}$, satisfying the usual conditions.
\item $W= (W_t, \mathcal{F}_{t})_{t\in[0,T]}$ is a Bm, $B^H= (B^H_t)_{t\in[0,T]}$ is a fBm and  $X= (X_t, \mathcal{F}_{t})_{t\in[0,T]}$ is a continuous and $\mathcal{F}_{t}$-adapted process satisfying a.s.~the Equation (\ref{eq:C2}) for some increasing
sequence of stopping times $(T_{R})_{R>0}$ such that $\lim_{R\rightarrow\infty}T_{R}=T$
a.s..
\end{enumerate} 
\end{defn}

\begin{defn}[Pathwise uniqueness]
We say that pathwise uniqueness holds for Equation (\ref{eq:1-1})
if, whenever $(X,W,B^{H})$ and $(\tilde{X},W,B^{H})$
are two weak solutions of Equation (\ref{eq:1-1}) defined on the
same probability space $\big(\Omega,\mathcal{F},(\mathcal{F}_{t})_{t\in[0,T]},P\big)$
then $X$ and $\tilde{X}$ are indistinguishable.
\end{defn}

\section{Pathwise uniqueness}\label{sec:pathunique}

In this section we investigate the pathwise uniqueness of a solution for Equation~\eqref{eq:1-1},  cf.~Theorem~\ref{pathunique} below, where we make use of the so-called Bihari's type lemma (see Lemma~\ref{lem:Bihari} in Appendix).

Let $X$ be a solution of Equation (\ref{eq:1-1}).
For $R>0$, we define the following stopping time
\[
T_{R}  :=  \inf\Big\{ t\geq0,\left\Vert B^{H}\right\Vert _{1-\alpha,\infty,t}\geq R\Big\} \wedge T,
\]

 For every positive constant $R$, we define the stochastic processes $X_R$ by
\[
X_{R}(t):=X(t\wedge T_{R}),\quad t\in [0,T].
\]

Then it is easy to see that the following equation

\begin{eqnarray*}
X_{R}(t) & =x_{0}+ & \int_{0}^{t\wedge T_{R}}b(s,X(s))\, ds+\int_{0}^{t\wedge T_{R}}\sigma_{W}(s,X(s))dW(s)\\
\\
 &  & +\int_{0}^{t\wedge T_{R}}\sigma_{H}(s,X(s))dB^{H}(s)
\end{eqnarray*}
 holds almost surely. We have the following Lemma.
\begin{lem}
\label{lem: norm est}For any integer $N\geq1$ and $R>0$, there exists a positive constant $C_{N}$ such that 
\[
\sup_{t\in[0,T]}\mathbb{E}\left[\|X_{R}\|_{\alpha,t}^{2N}\right]\leq C_{N}R^{2N}.
\]
\end{lem}
\begin{proof} 
Along the proof $C_N$ will denote a generic positive constant, which may vary from line to line and  may depend on $N$ and other parameters of the problem. It follows from the convexity of $x^{2N}$ that 
\begin{eqnarray*}
\mathbb{E}\big[\|X_{R}\|_{\alpha,t}^{2N}\big] & \leq & C_{N}\left\{ \left|x_{0}\right|^{2N}+\mathbb{E}\left[\left\|\int_{0}^{\cdot\wedge T_{R}}b(s,X(s))\, ds\right\|_{\alpha,t}^{2N}\right]\right.\\
 &  & +\mathbb{E}\left[\left\|\int_{0}^{\cdot\wedge T_{R}}\sigma_{W}(s,X(s))\, dW(s)\right\|_{\alpha,t}^{2N}\right]\\
 &  & +\left.\mathbb{E}\left[\left\|\int_{0}^{\cdot\wedge T_{R}}\sigma_{H}(s,X(s))\, dB^{H}(s)\right\|_{\alpha,t}^{2N}\right]\right\} \\
 & = & C_{N}\big(\left|x_{0}\right|^{2N}+A_{1}+A_{2}+A_{3}\big).
\end{eqnarray*}
Furthermore we have
\begin{eqnarray*}
&&\left\|\int_{0}^{.\wedge T_{R}}b(s,X(s))\, ds\right\|_{\alpha,t} \\ & &\leq \int_{0}^{t\wedge T_{R}}\left|b(s,X(s))\right|\, ds+\int_{0}^{t}(t-s)^{-\alpha-1}\int_{s\wedge T_{R}}^{t\wedge T_{R}}\left|b(u,X(u))\right|\, du\, ds\\
\\
 && \leq  \int_{0}^{t}\left|b(s\wedge T_{R},X(s\wedge T_{R}))\right|\, ds\\
 &  & +\int_{0}^{t}(t-s)^{-\alpha-1}\int_{s}^{t}\left|b(u\wedge T_{R},X(u\wedge T_{R}))\right|\, du\, ds\\
\\
 & &\leq  \int_{0}^{t}\left|b(s\wedge T_{R},X(s\wedge T_{R}))\right|\, ds\\
 &  & +\dfrac{1}{\alpha}\int_{0}^{t}(t-r)^{-\alpha}\left|b(r\wedge T_{R},X(r\wedge T_{R}))\right|\, dr\\
\\
 & &\leq  C_{\alpha,T}\int_{0}^{t}(t-r)^{-\alpha}\left|b(r\wedge T_{R},X(r\wedge T_{R}))\right|\, dr
\end{eqnarray*}
where $C_{\alpha,T}$ is a constant depending on $\alpha$ and $T$.
Using the linear growth assumption in ({\bf{H.1.1}}),  H{\"o}lder's inequality and the fact that $\alpha <\frac12$, we obtain
\begin{eqnarray*}
A_{1} & \leq & C_{N}\,\mathbb{E}\left[\left(1+\int_{0}^{t}\dfrac{\left|X_{R}(s)\right|}{(t-s)^{\alpha}}\, ds\right)^{2N}\right]\\
 & \leq & C_{N}\,\mathbb{E}\left[\left(1+\int_{0}^{t}\left|X_{R}(s)\right|^{2}\, ds\right)^{N}\right]\\
 & \leq & C_{N}\,\left(1+\int_{0}^{t}\mathbb{E}\left[\left|X_{R}(s)\right|^{2N}\right]\, ds\right).
\end{eqnarray*}
We have also that
\begin{eqnarray*}
A_{2} & \leq & C_{N}\mathbb{E}\left[\left|\int_{0}^{t\wedge T_{R}}\sigma_{W}(s,X(s))dW(s)\right|^{2N}\right]
\\ &&+C_{N}\mathbb{E}\left[\left(\int_{0}^{t}{(t-s)^{-\alpha-1}}{\left|\int_{s\wedge T_{R}}^{t\wedge T_{R}}\sigma_{W}(u,X(u))\, dW(u)\right|}\, ds\right)^{2N}\right]\\
\\
 & = & A_{21}+A_{22}.
\end{eqnarray*}

For $A_{21}$, using the linear growth assumption in $\mathbf{(H.1.3)}$, the Burkh{\"o}lder
and H{\"o}lder inequalities, we obtain 

\begin{eqnarray*}
A_{21} & \leq & C_{N}\mathbb{E}\left[\int_{0}^{t\wedge T_{R}}\left|\sigma_{W}(s,X(s))\right|^{2N}ds\right]\\
\\
 & \leq & C_{N}\mathbb{E}\left[\int_{0}^{t}\left|\sigma_{W}(s\wedge T_{R},X(s\wedge T_{R}))\right|^{2N}ds\right]\\
\\
 & \leq & C_{N}\left(1+\int_{0}^{t}\mathbb{E}\left[\left|X_{R}(s)\right|^{2N}\right]ds\right)
\end{eqnarray*}

For $A_{22}$, again the Burkh{\"o}lder and H{\"o}lder inequalities give
\begin{eqnarray*}
A_{22} & \leq & C_{N}\left(\int_{0}^{t}\dfrac{ds}{(t-s)^{\alpha+\frac{1}{2}}}\right)^{2N-1}\int_{0}^{t}{(t-s)^{-\alpha-\frac{1}{2}-N}}{\mathbb{E}\left[\left|\int_{s\wedge T_{R}}^{t\wedge T_{R}}\sigma_{W}(u,X(u))\, dW(u)\right|^{2N}\right]}ds\\
 & \leq & C_{N}\dint_{0}^{t}{(t-s)^{-\alpha-\frac{3}{2}}}{\mathbb{E}\left[\int_{s\wedge T_{R}}^{t\wedge T_{R}}\left|\sigma_{W}(u,X(u))\right|^{2N}\, du\right]} ds\\
 & \leq & C_{N}\int_{0}^{t}{(t-s)^{-\alpha-\frac{3}{2}}}{\mathbb{E}\left[\int_{s}^{t}\left|\sigma_{W}(u\wedge T_{R},X(u\wedge T_{R}))\right|^{2N}\, du\right]}\, ds.
\end{eqnarray*}
Applying now Fubini\textquoteright s theorem and using the growth
assumption in $\mathbf{(H.1.3)}$, we obtain
\[
A_{22}\leq C_{N}\left(\int_{0}^{t}(t-s)^{-\alpha-\frac{1}{2}}\left(1+\mathbb{E}\left[\left|X_{R}(s)\right|^{2N}\right]\right)\, ds\right).
\]
Thus 
\[
A_{2}\leq C_{N}\left(1+\int_{0}^{t}(t-s)^{-\alpha-\frac{1}{2}}\mathbb{E}\left[\left|X_{R}(s)\right|^{2N}\right]\, ds\right).
\]
Let us remark that, for $t\in [0,T]$, we have 
\begin{equation}
\int_{0}^{t\wedge T_{R}}\sigma_{H}(s,X(s))\, dB^{H}(s)=\int_{0}^{t}\sigma_{H}(s\wedge T_{R},X(s\wedge T_{R}))\, dB^{H}(s\wedge T_{R}).\label{stopfbmint}
\end{equation}
Then it follows from Proposition \ref{prop:Nua Ras 2} (jj), in the Appendix,
that 
\[
A_{3}\leq C_{N}R^{2N}\int_{0}^{t}\left((t-s)^{-2\alpha}+s^{-\alpha}\right)\left(1+\mathbb{E}\left[\|X_{R}\|_{\alpha,s}^{2N}\right]\right)\, ds.
\]
Putting all the estimates obtained for $A_{1}$, $A_{2}$ 
and $A_{3}$ together, we obtain

\begin{equation}
\mathbb{E}\left[\|X_{R}\|_{\alpha,t}^{2N}\right]  \leq  C_{N}\left|x_{0}\right|^{2N}+C_{N}(1+R^{2N}) \int_{0}^{t}\varphi(t,s)\mathbb{E}\left[\|X_{R}\|_{\alpha,s}^{2N}\right]\, ds, \label{eq:est esp-1}
\end{equation}
where  
\begin{eqnarray*}
\varphi(t,s) & := & s^{-\alpha}+(t-s)^{-\alpha-1/2}.
\end{eqnarray*}
Therefore, since the right  hand side of Equation (\ref{eq:est esp-1})
is an increasing function of $t$, we have
\begin{eqnarray*}
\sup_{0\leq s\leq t}\mathbb{E}\big[\|X_{R}\|_{s}^{2N}\big] & \leq & C_{N}\left|x_{0}\right|^{2N}+C_{N}\big(1+R^{2N}\big) \int_{0}^{t}\varphi(s,t)\sup_{0\leq u\leq s}\mathbb{E}\big[\|X_{R}\|_{\alpha,u}^{2N}\big]\, ds.
\end{eqnarray*}
As a consequence, by the Gronwall type lemma (Lemma 7.6 in \cite{NR}),
we deduce the desired estimate.\end{proof}
Let $X$ and $Y$ be two solutions of Equation (\ref{eq:1-1}) defined on the
same probability space $(\Omega,\mathcal{F},(\mathcal{F}_{t\in[0,T]}),P)$.
For $M>0$, we define the following stopping time
\[
\tau_{M}  :=  \inf\big\{t:\|X\|_{\alpha,t}  \vee \|Y\|_{\alpha,t}>M \big\}\wedge T.
\]
Now for every positive constants $R$ and $M$, we define the stochastic processes $X_{R,M}$ (resp. $Y_{R,M}$)  by
\[
X_{R,M}(t):=X(t\wedge T_{R}\wedge\tau_{M}),\quad t\in [0,T],
\]
(resp. $ Y_{R,M}(t):=Y(t\wedge T_{R}\wedge\tau_{M}),\quad t\in [0,T] $).

\begin{lem}
Under Hypotheses $\mathbf{(H.1)}$ and $\mathbf{(H.2)}$, there exists a positive constant $C_{R,M}$
such that for $t\in[0,T]$,
\begin{multline}
\mathbb{E}\left[\|X_{R,M}-Y_{R,M}\|_{\alpha,t}^{2}\right] \\
\leq C_{R,M}\int_{0}^{t}\varphi(s,t)\Big[\mathbb{E}\left[\|X_{R,M}-Y_{R,M}\|_{\alpha,s}^{2}\right]
\label{eq:differ est} 
+\varrho\left(\mathbb{E}\left[\|X_{R,M}-Y_{R,M}\|_{\alpha,s}^{2}\right]\right)\Big]\, ds. 
\end{multline}

\end{lem}
\begin{proof}
The proof of this result is long and technical. It is divided into
several parts. First we have

\begin{eqnarray*}
X_{R,M}(t)-Y_{R,M} (t)& = & \int_{0}^{t\wedge T_{R}\wedge\tau_{M}}(b(s,X(s))-b(s,Y(s)))\, ds\\ 
& &+\int_{0}^{t\wedge T_{R}\wedge\tau_{M}}(\sigma_{W}(s,X(s))-\sigma_{W}(s,Y(s)))\, dW(s)\\
 &  & +\int_{0}^{t\wedge T_{R}\wedge\tau_{M}}(\sigma_{H}(s,X(s))-\sigma_{H}(s,Y(s)))\, dB^{H}(s)\\
\\
 & = & B_{1}(t\wedge T_{R}\wedge\tau_{M})+B_{2}(t\wedge T_{R}\wedge\tau_{M})+B_{3}(t\wedge T_{R}\wedge\tau_{M}).
\end{eqnarray*}
It follows that
\begin{eqnarray*}
&&\|X_{R,M}-Y_{R,M}\|_{\alpha,t}^{2}\\&& \leq   3\left(\|B_{1}(\cdot\wedge T_{R}\wedge\tau_{M})\|_{\alpha,t}^{2} +\|B_{2}(\cdot\wedge T_{R}\wedge\tau_{M})\|_{\alpha,t}^{2} +\|B_{3}(\cdot\wedge T_{R}\wedge\tau_{M})\|_{\alpha,t}^{2} \right).
\end{eqnarray*} 

We have to estimated $\|B_{i}(\cdot\wedge T_{R}\wedge\tau_{M})\|_{\alpha,t}^{2}$, $i\in\{1,2,3\}$.  For the sake of conciseness, we define $$\Delta(f)(s)=f(s,X(s))-f(s,Y(s)), \quad f\in \{b,\sigma_{W},\sigma_{H}\}.$$
\begin{description}
\item[Step~1: $B_1$.] Using simple estimations it is easy to see that

\begin{eqnarray*}
\|B_{1}(\cdot\wedge T_{R}\wedge\tau_{M})\|_{\alpha,t} &\leq & \int_{0}^{t\wedge T_{R}\wedge\tau_{M}}\left|\Delta(b)(s)\right|\, ds\\
& +& \int_{0}^{t}(t-s)^{-\alpha-1}\int_{s\wedge T_{R}\wedge\tau_{M}}^{t\wedge T_{R}\wedge\tau_{M}}\left|\Delta(b)(u)\right|\, du\, ds
 \\
& \leq & \int_{0}^{t}\left|\Delta(b)(s\wedge T_{R}\wedge\tau_{M})\right|\, ds \\
& +&\int_{0}^{t}(t-s)^{-\alpha-1}\int_{s}^{t}\left|\Delta(b)(u\wedge T_{R}\wedge\tau_{M})\right|\, du\, ds\\
& \leq &
C_{\alpha, T}\int_{0}^{t}(t-r)^{-\alpha}\left|\Delta(b)(r\wedge T_{R}\wedge\tau_{M})\right|\, dr.  
\end{eqnarray*}

We use the fact that $\alpha<\frac{1}{2}$, H{\"o}lder inequality and hypothesis $\mathbf{(H.1.2)}$
to obtain 
\begin{eqnarray*}
\|B_{1}(.\wedge T_{R}\wedge\tau_{M})\|_{\alpha,t}^{2} 
 & \leq & C_{\alpha, T}^{2}\int_{0}^{t}\dfrac{\left|\Delta(b)(s\wedge T_{R}\wedge\tau_{M})\right|^{2}}{\left(t-s\right)^{\alpha}}\, ds\\
 & \leq & C_{\alpha, T}^{2}\int_{0}^{t}\dfrac{\varrho\left(|X_{R,M}(s)-Y_{R,M}(s)|^{2}\right)}{\left(t-s\right)^{\alpha}}\, ds\\
 & \leq & C_{\alpha, T}^{2}\int_{0}^{t}\varphi(s,t)\varrho\left(\|X_{R,M}-Y_{R,M}\|_{\alpha,s}^{2}\right)\, ds.
\end{eqnarray*}
 \item[Step~2: $B_3$.]
If $1-H<\alpha<\min\left(\beta,1/2\right)$ , we have from Proposition
4.3 in \cite{NR} (see  Proposition \ref{prop:Nua Ras 1} (ii)
in the Appendix) 
that 
\begin{multline*}
\|B_{3}(\cdot\wedge T_{R}\wedge\tau_{M})\|_{\alpha,t}^{2}  \leq  C R^{2} \bigg(\int_{0}^{t}\left((t-s)^{-2\alpha}+s^{-\alpha}\right)\|\Delta(\sigma_{H})(\cdot\wedge T_{R}\wedge\tau_{M})\|_{\alpha,s}\, ds\bigg)^{2}.
\end{multline*}
Now using the assumptions $\mathbf{(H.2)}$ and Lemma 7.1 in Nualart
Rascanu \textcolor{magenta}{\cite{NR}} we obtain
\begin{eqnarray*}
& & \left|\sigma_{H}(t,x_{1})-\sigma_{H}(s,x_{2})-\sigma_{H}(t,y_{1})+\sigma_{H}(s,y_{2})\right| \\
& \leq & K\left|x_{1}-x_{2}-y_{1}+y_{2}\right|+K\left|x_{1}-y_{1}\right|\left|t-s\right|^{\beta}\\ 
& & +K\left|x_{1}-y_{1}\right|\left(\left|x_{1}-x_{2}\right|+\left|y_{1}-y_{2}\right|\right).
\end{eqnarray*}
Therefore
\begin{eqnarray*}
&  &\bigg|\sigma_{H}(t\wedge T_{R}\wedge\tau_{M},X_{R,M}(t))-\sigma_{H}(s\wedge T_{R}\wedge\tau_{M},X_{R,M}(s))\\
\\& &-\sigma_{H}(t\wedge T_{R}\wedge\tau_{M},Y_{R,M}(t))+\sigma_{H}(s\wedge T_{R}\wedge\tau_{M},Y_{R,M}(s))\bigg| \\ 
\\ &\leq & K\Big[\left|X_{R,M}(t)-X_{R,M}(s)-Y_{R,M}(t)+Y_{R,M}(s)\right|+K\left|X_{R,M}(t)-Y_{R,M}(t)\right|\left|t-s\right|^{\beta}\\
\\ & & +\left|X_{R,M}(t)-Y_{R,M}(t)\right|\left(\left|X_{R,M}(t)-X_{R,M}(s)\right|+\left|Y_{R,M}(t)-Y_{R,M}(s)\right|\right)\Big].
\end{eqnarray*}
Thus we have 
\begin{eqnarray*}
& &\|\Delta(\sigma_{H})(.\wedge T_{R}\wedge\tau_{M})\|_{\alpha,t} \\ 
\\
& \leq & K\left[\left|X_{R,M}(t)-Y_{R,M}(t)\right|+\int_{0}^{t}\dfrac{\left|X_{R,M}(t)-X_{R,M}(s)-Y_{R,M}(t)+Y_{R,M}(s)\right|}{\left(t-s\right)^{\alpha+1}}\, ds\right.\\
&  & +\left|X_{R,M}(t)-Y_{R,M}(t)\right|\left(\int_{0}^{t}\dfrac{ds}{\left(t-s\right)^{\alpha-\beta+1}}+\int_{0}^{t}\dfrac{\left|X_{R,M}(t)-X_{R,M}(s)\right|}{\left(t-s\right)^{\alpha+1}}\, ds\right.\\
&  & \qquad\qquad\qquad\qquad\qquad\qquad+\left.\left.\int_{0}^{t}\dfrac{\left|Y_{R,M}(t)-Y_{R,M}(s)\right|}{\left(t-s\right)^{\alpha+1}}\, ds\right)\right].
\end{eqnarray*}
Now it is easy to see that 
\begin{eqnarray*}
& & \|B_{3}(\cdot\wedge T_{R}\wedge\tau_{M})\|_{\alpha,t}^{2}  \\
& \leq &CR^{2}\int_{0}^{t}\left((t-s)^{-2\alpha}+s^{-\alpha}\right)\left(1+\|X_{R,M}\|_{\alpha,s}^{2}+\|Y_{R,M}\|_{\alpha,s}^{2}\right)\|X_{R,M}-Y_{R,M}\|_{\alpha,s}^{2}\, ds\\
 & \leq & CR^{2}M^{2}\int_{0}^{t}\varphi(s,t)\|X_{R,M}-Y_{R,M}\|_{\alpha,s}^{2}\, ds.
\end{eqnarray*}
\item[Spet~3: $B_2$.]
Till now we have made estimates for pathwise integrals.
As $B_{2}$ is a stochastic integral we need to use martingale type
inequality. First we have 
\[
\|B_{2}(\cdot\wedge T_{R}\wedge\tau_{M})\|_{\alpha,t}^{2}
\leq 2\left(|B_{2}(t\wedge T_{R}\wedge\tau_{M})|^{2}+\big(\tilde{B}_2(t)\big)^{2}\right),
\]
where 
\[
\tilde{B}_{2}(t):=\int_{0}^{t}\frac{\left|\int_{s\wedge T_{R}\wedge\tau_{M}}^{t\wedge T_{R}\wedge\tau_{M}}\Delta(\sigma_{W})(u)\, dW(u)\right|}{(t-s)^{\alpha+1}}\, ds.
\]

It then follows from Burkh{\"o}lder inequality and assumption $\mathbf{(H.1.4)}$
that
\begin{eqnarray*}
\mathbb{E}\big(|B_{2}(t\wedge T_{R}\wedge\tau_{M})|^{2}\big) 
 &\leq & \mathbb{E}\left(\int_{0}^{t}|\Delta(\sigma_{W})(s\wedge T_{R}\wedge\tau_{M})|^{2}\, ds\right)\\
  &\leq & C\,\mathbb{E}\left(\int_{0}^{t}\varrho\left(|X_{R,M}(s)-Y_{R,M}(s)|^{2}\right)\, ds\right)\\
  &\leq & C\,\mathbb{E}\left(\int_{0}^{t}\varrho\left(\|X_{R,M}-Y_{R,M}\|_{\alpha,s}^{2}\right)\, ds\right).
\end{eqnarray*}
$\tilde{B_{2}}$: Using H{\"o}lder's inequality and Fubini's theorem we have
\begin{eqnarray*}
 \mathbb{E}\left[\big|\tilde{B}_{2}(t)\big|^{2}\right]
 & \leq & C\,\mathbb{E}\left[\int_{0}^{t}{(t-s)^{-\frac{3}{2}-\alpha}}{\left|\int_{s\wedge T_{R}\wedge\tau_{M}}^{t\wedge T_{R}\wedge\tau_{M}}\Delta(\sigma_{W})(u)\, dW(u)\right|^{2}}\, ds\right]\\
\\
 & \leq & C\int_{0}^{t}{(t-s)^{-\frac{3}{2}-\alpha}}{\mathbb{E}\left[\left|\int_{s\wedge T_{R}\wedge\tau_{M}}^{t\wedge T_{R}\wedge\tau_{M}}\Delta(\sigma_{W})(u)\, dW(u)\right|^{2}\right]} ds.
\end{eqnarray*}
Using the same techniques as in the estimation of $I_{2}$ we have
\begin{eqnarray*}
 \mathbb{E}\left[\left|\int_{s\wedge T_{R}\wedge\tau_{M}}^{t\wedge T_{R}\wedge\tau_{M}}\Delta(\sigma_{W})(u)\, dW(u)\right|^{2}\right]
 & \leq & \mathbb{E}\left[\int_{s}^{t}|\Delta(\sigma_{W})(u\wedge T_{R}\wedge\tau_{M})|^{2}\, du\right]\\
 & \leq & C\,\mathbb{E}\left[\int_{s}^{t}\varrho\left(\|X_{R,M}-Y_{R,M}\|_{\alpha,u}^{2}\right)du\right].
\end{eqnarray*}
Then, it follows that 
\begin{eqnarray*}
\mathbb{E}\left[\big|\tilde{B}_{2}(t)\big|^{2}\right]
 & \leq & C\int_{0}^{t}{(t-s)^{-\frac{3}{2}-\alpha}}{\mathbb{E}\left[\int_{s}^{t}\varrho\left(\|X_{R,M}-Y_{R,M}\|_{\alpha,u}^{2}\right)du\right]}ds.
 \end{eqnarray*}
Consequently 
\[
\mathbb{E}\big[\|B_{2}\|_{\alpha,t}^{2}\big]\\
\leq C\int_{0}^{t}\varphi(s,t)\mathbb{E}\left[\varrho\left(\|X_{R,M}-Y_{R,M}\|_{\alpha,s}^{2}\right)\right]ds.
\]
\item[Step~4:]
Combining all estimates, leads to
\begin{eqnarray*}
&& \mathbb{E}\left[\|X_{R,M}-Y_{R,M}\|_{\alpha,s}^{2}\right]
\\ && \leq
C_{M,R}\int_{0}^{t}\varphi(s,t)\mathbb{E}\Big[\|X_{R,M}-Y_{R,M}\|_{\alpha,s}^{2} +\varrho\left(\|X_{R,M}-Y_{R,M}\|_{\alpha,s}^{2}\right)\Big]ds.
\end{eqnarray*}
Since $\varrho$ is concave, Jensen's inequality gives
\begin{eqnarray*}
&&
\mathbb{E}\left[\|X_{R,M}-Y_{R,M}\|_{\alpha,s}^{2}\right]
\\ && \leq
C_{M,R}\int_{0}^{t}\varphi(s,t)\Big[\mathbb{E}\left[\left(\|X_{R,M}-Y_{R,M}\|_{\alpha,s}^{2}\right)\right]
 +\varrho\left(\mathbb{E}\left[\|X_{R,M}-Y_{R,M}\|_{\alpha,s}^{2}\right]\right)\Big]ds.
\end{eqnarray*}
\end{description}
This concludes the proof.
\end{proof}
\begin{thm}[Pathwise uniqueness]\label{pathunique}
Let $1-H<\alpha<\min\left(\beta,1/2\right)$.
Then, under hypotheses $\mathbf{(H.1)}$ and $\mathbf{(H.2)}$, the pathwise uniqueness property
holds for Equation (\ref{eq:1-1}).
\end{thm}
\begin{proof}
It is simple to see that the function $\tilde{\varrho}(u)=u+\varrho(u)$
is a concave increasing function from $\mathbb{R}_{+}$ to $\mathbb{R}_{+}$
such that $\tilde{\varrho}(0)=0$ and $\tilde{\varrho}(u)>0$ for
$u>0$. On the other hand, we have $\varrho(u)\geq\varrho(1)u$ for
$0\leq u\leq1$. Then 
\[
\int_{0^{+}}\dfrac{du}{\tilde{\varrho}^{q}(u^{1/q})}\geq\left(\dfrac{\varrho(1)}{1+\varrho(1)}\right)^{q}\int_{0^{+}}\dfrac{du}{\varrho^{q}(u^{1/q})}=\infty.
\]
 Therefore, the condition (\ref{eq:rho est}) is satisfied for the
function $\tilde{\varrho}$. Consequently, we can apply Lemma \ref{lem:Bihari}
in the Appendix to the inequality (\ref{eq:differ est}) to obtain
\[
\|X_{R,M}-Y_{R,M}\|_{\alpha,t}^{2}=0,\,\, a.s.
\]
 This implies $X(t)=Y(t)$ a.s.~for all $t< T_{R}\wedge\tau_{M}$.
By letting $M\rightarrow\infty$ we get, by Lemma \ref{lem: norm est},
$X(t)=Y(t)\,\, a.s.$ for all $t<T_{R}$. Using the that fact that the random variable $\| B^{H}\|_{1-\alpha,\infty,t}$
has moments of all orders, see Lemma~7.5 in Nualart  and Rascanu \cite{NR},	it is not difficult that almost surely $T_{R} = T$ for $R$ large enough. This concludes the proof.

\end{proof}

\section{Euler Approximation scheme}\label{sec:Euler}
\textcolor{magenta}{}
In this section, we apply the Euler approximation procedure in order to obtain a weak solution of 
Equation~\eqref{eq:1-1}. Under the condition that pathwise uniqueness holds for Equation~\eqref{eq:1-1} we prove that the Euler approximation converges to a process which is a strong solution of the SDE~\eqref{eq:1-1}, see Theorem~\ref{thm:Eulerstrong} below.

Let $0=t_{0}^{n}<t_{1}^{n}<\cdots<t_{i}^{n}<\cdots<t_{n}^{n}=T$
 be a sequence of partitions of $[0,T]$ such that 
\[
\underset{0\leq i\leq n-1}{\sup}\left|t_{i+1}^{n}-t_{i}^{n}\right|\rightarrow 0,\quad \mbox{as}\quad n\rightarrow\infty.
\]
 We define Euler's approximations as the
process $X^{n}$, $n\in\mathbb{N}$, satisfying
\begin{eqnarray}
X^{n}(t)& = &  x_{0}+\int_{0}^{t}b(k_{n}(s),X(k_{n}(s)))\, ds+\int_{0}^{t}\sigma_{W}(k_{n}(s),X(k_{n}(s)))\, dW(s)\nonumber \\
 & &+\int_{0}^{t}\sigma_{H}(k_{n}(s),X(k_{n}(s)))\, dB^{H}(s),
\label{eq:app sch}
\end{eqnarray}
where $k_{n}(t):=t_{i}^{n}$ if $t\in\left[t_{i}^{n},t_{i+1}^{n}\right)$ and $t\in [0,T]$.
For every positive constant $R$ we define the family of stochastic
processes by
\[
X_{R}^{n}(t):=X^{n}(t\wedge T_{R}), \quad t\in [0,T]. 
\]
Then it is easy to see that the process $X_{R}^{n}$
satisfies, a.s., the following
\begin{eqnarray*}
X_{R}^{n}(t) & =x_{0}+ & \int_{0}^{t\wedge T_{R}}b(k_{n}(s),X^{n}(k_{n}(s)))\, ds+\int_{0}^{t\wedge T_{R}}\sigma_{W}(k_{n}(s),X^{n}(k_{n}(s)))\, dW(s)\\
 &  & +\int_{0}^{t\wedge T_{R}}\sigma_{H}(k_{n}(s),X^{n}(k_{n}(s)))\, dB^{H}(s).
\end{eqnarray*}
We obtain for any integer $N\geq1$ 
\begin{lem}
\label{lem:tight} Suppose that Assumptions $\mathbf{(H.1)}$ and $\mathbf{(H.2)}$ hold. Then, for all $n\in\mathbb{N}$, $N\in\mathbb{N}^*$ and $R>0$,  there exists a positive constant $C_{N,R}$ such that
\begin{equation}
\sup_{t\in[0,T]}\mathbb{E}\left[\|X_{R}^{n}\|_{\alpha,t}^{2N}\right]\leq C_{N,R}. \label{app est}
\end{equation}
Moreover, we also have for all $s,t\in[0,T]$,
\begin{equation}
\mathbb{E}\left[\left|X_{R}^{n}(t)-X_{R}^{n}(s)\right|^{2N}\right]\leq C_{N,R}\left|t-s\right|^{N}. \label{app inc est}
\end{equation}
\end{lem}
\begin{proof}
It follows from the convexity of $x^{2N}$ that 
\begin{eqnarray*}
\mathbb{E}\left[\|X_{R}^{n}\|_{\alpha,t}^{2N}\right] & \leq & C_{N}\left\{ \left|x_{0}\right|^{2N}+\mathbb{E}\left[\left\|\int_{0}^{\cdot\wedge T_{R}}b(k_{n}(s),X^{n}(k_{n}(s)))\, ds\right\|_{\alpha,t}^{2N}\right]\right.\\
 &  & +\mathbb{E}\left[\left\|\int_{0}^{\cdot\wedge T_{R}}\sigma_{W}(k_{n}(s),X^{n}(k_{n}(s)))dW(s)\right\|_{\alpha,t}^{2N}\right]\\
 &  & +\left.\mathbb{E}\left[\left\|\int_{0}^{\cdot\wedge T_{R}}\sigma_{H}(k_{n}(s),X^{n}(k_{n}(s)))dB^{H}(s)\right\|_{\alpha,t}^{2N}\right]\right\} \\
 & = & C_{N}\left(\left|x_{0}\right|^{2N}+I_{1}+I_{2}+I_{3}\right).
\end{eqnarray*}
Using the same estimations as in the proof of Lemma \ref{lem: norm est}, we obtain
\begin{eqnarray*}
I_{1} &  \leq & C_{N}\left(1+\int_{0}^{t}\mathbb{E}\left[\left|X^{n}(k_{n}(s)\wedge T_{R})\right|^{2N}\right]ds\right)\\
&  \leq & C_{N}\left(1+\int_{0}^{t}\mathbb{E}\left[\left|X_{R}^{n}(k_{n}(s))\right|^{2N}\right]ds\right).
\end{eqnarray*}
\begin{eqnarray*}
I_{2} & \leq & C_{N}\mathbb{E}\left[\left|\int_{0}^{t\wedge T_{R}}\sigma_{W}(k_{n}(s),X^{n}(k_{n}(s)))\,dW(s)\right|^{2N}\right] \\ 
&  & + C_{N}\mathbb{E}\left[\left(\int_{0}^{t}{(t-s)^{-\alpha-1}}{\left|\int_{s\wedge T_{R}}^{t\wedge T_{R}}\sigma_{W}(k_{n}(s),X^{n}(k_{n}(s)))\, dW(u)\right|}\, ds\right)^{2N}\right]\\
\\
 & = & I_{21}+I_{22}.
\end{eqnarray*}
For $I_{21}$, using the linear growth assumption in $\mathbf{(H.1.3)}$, Burkh{\"o}lder's
and H{\"o}lder's inequalities, we obtain 

\begin{eqnarray*}
I_{21} & \leq & C_{N}\mathbb{E}\left[\int_{0}^{t\wedge T_{R}}\left|\sigma_{W}(k_{n}(s),X^{n}(k_{n}(s)))\right|^{2N}ds\right]\\
\\
 & \leq & C_{N}\mathbb{E}\left[\int_{0}^{t}\left|\sigma_{W}(k_{n}(s)\wedge T_{R},X^{n}(k_{n}(s)\wedge T_{R}))\right|^{2N}ds\right]\\
\\
 & \leq & C_{N}\left(1+\int_{0}^{t}\mathbb{E}\left[\left|X_{R}^{n}(k_{n}(s))\right|^{2N}\right]\, ds\right).
\end{eqnarray*}
For $I_{22}$ , again the Burkh{\"o}lder and H{\"o}lder inequalities give 
\begin{eqnarray*}
I_{22}   & 
\leq &  C_{N}\left(\int_{0}^{t}\dfrac{ds}{(t-s)^{\alpha+\frac{1}{2}}}\right)^{2N-1}  \\ 
&& \times \int_{0}^{t}{(t-s)^{-\alpha-\frac{1}{2}-N}}{\mathbb{E}\left[\left|\int_{s\wedge T_{R}}^{t\wedge T_{R}}\sigma_{W}(k_{n}(s),X^{n}(k_{n}(s)))\, dW(u)\right|^{2N}\right]}ds\\
 & \leq & C_{N}\int_{0}^{t}{(t-s)^{-\alpha-\frac{3}{2}}}{\mathbb{E}\left[\int_{s\wedge T_{R}}^{t\wedge T_{R}}\left|\sigma_{W}(k_{n}(s),X^{n}(k_{n}(s)))\right|^{2N}du\right]} ds\\
 & \leq & C_{N}\int_{0}^{t}(t-s)^{-\alpha-\frac{3}{2}}}{\mathbb{E}\left[\int_{s}^{t}\left|\sigma_{W}(k_{n}(u)\wedge T_{R},X^{n}(k_{n}(u)\wedge T_{R}))\right|^{2N}du\right]ds.
\end{eqnarray*}
Applying now Fubini\textquoteright s theorem and using the growth
assumption in $\mathbf{(H.1.3)}$, we obtain
\[
I_{22}\leq C_{N}\int_{0}^{t}(t-s)^{-\alpha-\frac{1}{2}}\left(1+\mathbb{E}\left[\left|X_{R}^{n}(k_{n}(s))\right|^{2N}\right]\right)ds.
\]
Thus 
\[
I_{2}\leq C_{N}\left(1+\int_{0}^{t}(t-s)^{-\alpha-\frac{1}{2}}\mathbb{E}\left[\left|X_{R}^{n}(k_{n}(s))\right|^{2N}\right]ds\right).
\]

Let us remark that

\begin{eqnarray}
& &\int_{0}^{t\wedge T_{R}}\sigma_{H}(k_{n}(s)),X^{n}(k_{n}(s))))\, dB^{H}(s) \label{frac int stopped} \\
 && =  \int_{0}^{t}\sigma_{H}(k_{n}(s)\wedge T_{R},X^{n}(k_{n}(s)\wedge T_{R}))\, dB^{H}(s\wedge T_{R})\nonumber\\\notag
& & = \int_{0}^{t}\sigma_{H}(k_{n}(s)\wedge T_{R},X_{R}^{n}(k_{n}(s)))\, dB^{H}(s\wedge T_{R})  
\end{eqnarray}

Using (\ref{frac int stopped}) and Proposition \ref{prop:Nua Ras 2} (jj) in the Appendix we obtain 
\[
I_{3}\leq C_{N}R^{2N}\left(\int_{0}^{t} \left((t-s)^{-2\alpha}+s^{-\alpha}\right) \big(1+\mathbb{E}\left[\|X_{R}^{n}(k_{n}(\cdot))\|_{\alpha,s}\right]\big)ds\right)^{2N}.
\]
By Hölder's inequality  we have
\[
I_{3}\leq C_{N}R^{2N}\int_{0}^{t} \varphi(s,t) \left(1+\mathbb{E}\left[\|X_{R}^{n}(k_{n}(\cdot))\|_{\alpha,s}^{2N}\right]\right)\, ds.
\]
Putting all the estimates obtained for $I_{1}$, $I_{2}$ and $I_{3}$ together, we obtain
\begin{equation} \label{eq:est esp}
\mathbb{E}\left[\|X_{R}^{n}\|_{\alpha,t}^{2N}\right] \leq C_{N}\left|x_{0}\right|^{2N}+C_{N}\left(1+R^{2N}\right)
\int_{0}^{t}\varphi(s,t)\mathbb{E}\left[\|X_{R}^{n}(k_{n}(\cdot))\|_{\alpha,s}^{2N}\right]\, ds. 
\end{equation}
Therefore, since the right hand side of Equation (\ref{eq:est esp})
is an increasing function of $t$, we have
\begin{equation*}
\sup_{0\leq s\leq t}\mathbb{E}\left[\|X_{R}^{n}\|_{\alpha,s}^{2N}\right]  \leq  C_{N}\left|x_{0}\right|^{2N}+C_{N}\left(1+R^{2N}\right)\\
\int_{0}^{t}\varphi(s,t)\mathbb{E}\left[\underset{0\leq u\leq s}{\sup}\|X_{R}^{n}\|_{\alpha,u}^{2N}\right]\, ds.
\end{equation*}
As a consequence, by the Gronwall type lemma (cf. Lemma 7.6 in \cite{NR}), we deduce the  first estimate (\ref{app est}) of the lemma. Let us now prove the second estimate (\ref{app inc est}). We have 
\begin{eqnarray*}
&& X_{R}^{n}(t)- X_{R}^{n}(s) \\ && =  \int_{s\wedge T_{R}}^{t\wedge T_{R}}b(k_{n}(r),X^{n}(k_{n}(r)))\, dr+\int_{s\wedge T_{R}}^{t\wedge T_{R}}\sigma_{W}(k_{n}(r),X^{n}(k_{n}(r)))\, dW(r)\\
 &  & \quad+\int_{s\wedge T_{R}}^{t\wedge T_{R}}\sigma_{H}(k_{n}(r),X^{n}(k_{n}(r)))\, dB^{H}(r).
\end{eqnarray*}
Therefore
\begin{eqnarray*}
 \mathbb{E}\left[\left|X_{R}^{n}(t)-X_{R}^{n}(s)\right|^{2N}\right]   &  \leq & C_{N}\left\{ \mathbb{E}\left[\left|\int_{s\wedge T_{R}}^{t\wedge T_{R}}b(k_{n}(r),X^{n}(k_{n}(r)))\, dr\right|^{2N}\right]\right.\\
 &  & +\mathbb{E}\left[\left|\int_{s\wedge T_{R}}^{t\wedge T_{R}}\sigma_{W}(k_{n}(r),X^{n}(k_{n}(r)))\, dW(r)\right|^{2N}\right]\\
 &  & +\left.\mathbb{E}\left[\left|\int_{s\wedge T_{R}}^{t\wedge T_{R}}\sigma_{H}(k_{n}(r),X^{n}(k_{n}(r)))\, dB^{H}(r)\right|^{2N}\right]\right\} \\
 & = & C_{N}\left(J_{1}+J_{2}+J_{3}\right).
\end{eqnarray*}
Applying H{\"o}lder's inequality, the growth assumption ({\bf{H.1.1}}) and (\ref{app est}), we have
\begin{eqnarray*}
J_{1}   &  \leq & \mathbb{E}\left[\left(\int_{s}^{t}\left|b(k_{n}(r)\wedge T_{R},X^{n}(k_{n}(r)\wedge T_{R}))\right|\, dr\right)^{2N}\right]\\
 & \leq & C_{N} (t-s)^{2N-1}\int_{s}^{t}\mathbb{E}\left[\left|b(k_{n}(r)\wedge T_{R},X_{R}^{n}(k_{n}(r)))\right|^{2N}\right] dr\\
 & \leq  & C_{N} (t-s)^{2N}.
\end{eqnarray*}
By the H{\"o}lder and Burkh{\"o}lder inequalities and using (\ref{app est}), we obtain
\begin{eqnarray*}
J_{2}   &  \leq & C_{N} (t-s)^{N-1}\mathbb{E}\left[\int_{s\wedge T_{R}}^{t\wedge T_{R}}\left|\sigma_{W}(k_{n}(r),X^{n}(k_{n}(r)))\right|^{2N}\, dr\right] \\
 & \leq & C_{N} (t-s)^{N-1}\mathbb{E}\left[\int_{s}^{t}\left|\sigma_{W}(k_{n}(r)\wedge T_{R},X_{R}^{n}(k_{n}(r)))\right|^{2N}\, dr\right]  \\
 & \leq &   C_{N} (t-s)^{N}.
\end{eqnarray*}
Let us note that we obtain from (\ref{frac int est}) and the H{\"o}lder inequality
\[
\left|\int_{s}^{t}f(u)\, dB^{H}(u)\right|^{2N}\leq C_{N}R^{2N}(t-s)^{2N(1-\alpha)+2\alpha-1}\int_{s}^{t}\dfrac{\|f(r)\|_{\alpha}^{2N}}{(r-s)^{2\alpha}}\, dr.
\]
Combining this estimate and (\ref{frac int stopped}) we obtain
\begin{eqnarray*}
J_{3}   &  \leq & C_{N}R^{2N}(t-s)^{2N(1-\alpha)+2\alpha-1}\mathbb{E}\left[\int_{s}^{t}\dfrac{\|\sigma_{H}(k_{n}(r)\wedge T_{R},X_{R}^{n}(k_{n}(r)))\|_{\alpha}^{2N}}{(r-s)^{2\alpha}}\, dr\right].
\end{eqnarray*}
Using the Hölder inequality, assumption ({\bf{H.2}}) and  (\ref{app est}), we arrive at
\begin{eqnarray*}
J_{3}   &  \leq & C_{N}R^{2N}(t-s)^{2N(1-\alpha)+2\alpha-1}\mathbb{E}\left[\int_{s}^{t}\dfrac{1+\|X_{R}^{n}(k_{n}(r)))\|_{\alpha}^{2N}}{(r-s)^{2\alpha}}\, dr\right] \\  \\
 & \leq &   C_{N} (t-s)^{N}.
\end{eqnarray*}
All these estimates allow us to obtain
\[
\mathbb{E}\left[\left|X_{R}^{n}(t)-X_{R}^{n}(s)\right|^{2N}\right]\leq C_{N,R}\left|t-s\right|^{N}.
\]
The proof of Lemma \ref{lem:tight} is then completed.
\end{proof}
Now we are able to give the convergence result.
\begin{thm}\label{thm:Eulerstrong}
Assume that $\sigma_{W}$ and $b$ are continuous satisfying the linear
growth condition. Suppose moreover that $\sigma_{H}$ satisfies the assumption $\mathbf{(H.2)}$
and that for Equation (\ref{eq:1-1}) the pathwise uniqueness holds.
Then Euler's approximations $X^{n}(t)$ converge to a process $X(t)$
in probability, uniformly in $t$ in $[0,T]$. Furthermore $X(t)$
is the unique strong solution of Equation (\ref{eq:1-1}). \end{thm}
\begin{proof}
Fix $\eta<1/2$. We have from (\ref{app inc est}) in Lemma \ref{lem:tight} that $X_{R}^{n}$
is weakly relatively compact in $C_{0}^{\eta}$ for every $R$. We
want to deduce from this the weak compactness in $C_{0}^{\eta}$ of
$X^{n}$. Clearly it suffices to show that 
\[
\limsup_{R\rightarrow\infty}P\left[T_{R}\leq T\right]=0.
\]
This is a consequence of that fact that the random variable $\left\Vert B^{H}\right\Vert _{1-\alpha,\infty,t}$
has moments of all orders (see Lemma 7.5 in \cite{NR}). We now
take two subsequences $X^{l},X^{m}$ of the Euler's approximations
$X^{n}$. Then obviously $\left(X^{l},X^{m}\right)$ is a tight family
of processes in $C_{0}^{\eta}\times C_{0}^{\eta}$. By Skorokhod's
embedding theorem there exist a probability space $\big(\tilde{\Omega},\tilde{\mathcal{F}},\tilde{P}\big)$
and a sequence $\big(\tilde{X}^{l,n},\tilde{X}^{m,n},\tilde{B}^{n},\tilde{W}^{n}\big)$
with values in $C_{0}^{\eta}$ such that
\begin{enumerate}
\item The law of $\big(\tilde{X}^{l,n},\tilde{X}^{m,n},\tilde{B}^{n},\tilde{W}^{n}\big)$
and $\left(X^{l},X^{m},B^{H},W\right)$ coincide for every $n\in\mathbb{N}$.
\item There exist a subsequence $\big(\tilde{X}^{l(j)},\tilde{X}^{m(j)},\tilde{B}^{n(j)},\tilde{W}^{n(j)}\big)$
converging in $C_{0}^{\eta}$ to $\big(\hat{X},\hat{Y},\hat{B},\hat{W}\big)$
uniformly in $t$, $\tilde{P}$ a.s., that is 
\[
\lim_{j\rightarrow\infty}\Big(\|\tilde{X}^{m(j)}-\hat{X}\|_{\eta}+\|\tilde{X}^{\textcolor{magenta}{l}(j)}-\hat{Y}\|_{\eta}+\|\tilde{B}^{n(j)}-\hat{B}\|_{\eta}+\|\tilde{W}^{n(j)}-\hat{W}\|_{\eta}\Big)=0.
\]
\end{enumerate}
We obtain from Lemma 3.1 in  G\"{y}ongy and Krylov \cite{GK} and the convergence
of integrals with respect to fBms (5.7) in
Guerra and Nualart \cite{GN} that 
\begin{eqnarray*}
\lim_{j\rightarrow\infty} \int_{0}^{t}b\big(k_{l(j)}(s),\tilde{X}^{l(j)}(k_{l(j)}(s))\big)\, ds &=& \int_{0}^{t}b\big(s,\hat{X}(s)\big)\, ds\\
\lim_{j\rightarrow\infty}  \int_{0}^{t}\sigma_{W}\big(k_{l(j)}(s),\tilde{X}^{l(j)}(k_{l(j)}(s))\big)\,d\tilde{W}^{n(j)}(s) &=& \int_{0}^{t}\sigma_{W}\big(s,\hat{X}(s)\big)\,d\hat{W}(s)\\
\lim_{j\rightarrow\infty} \int_{0}^{t}\sigma_{H}\big(k_{l(j)}(s),\tilde{X}^{l(j)}(k_{l(j)}(s))\big)\,d\tilde{B}^{n(j)}(s) &=& \int_{0}^{t}\sigma_{H}\big(s,\hat{X}(s)\big)\,d\hat{B}(s),
\end{eqnarray*}
and 
\begin{eqnarray*}
\lim_{j\rightarrow\infty} \int_{0}^{t}b\big(k_{m(j)}(s),\tilde{X}^{m(j)}(k_{m(j)}(s))\big)\, ds &=& \int_{0}^{t}b\big(s,\hat{Y}(s)\big)\, ds\\
\lim_{j\rightarrow\infty} \int_{0}^{t}\sigma_{W}\big(k_{m(j)}(s),\tilde{X}^{m(j)}(k_{m(j)}(s))\big)\, d\tilde{W}^{n(j)}(s) &=& \int_{0}^{t}\sigma_{W}\big(s,\hat{Y}(s)\big)\, d\hat{W}(s)\\
\lim_{j\rightarrow\infty} \int_{0}^{t}\sigma_{H}\big(k_{m(j)}(s),\tilde{X}^{m(j)}(k_{m(j)}(s))\big)\, d\tilde{B}^{n(j)}(s) &=& \int_{0}^{t}\sigma_{H}\big(s,\hat{Y}(s)\big)\, d\hat{B}(s),
\end{eqnarray*}
in probability, and uniformly in $t\in[0,T]$. Therefore, the processes
$\hat{X},\hat{Y}$ satisfy the same SDE~\eqref{eq:1-1}, on $(\tilde{\Omega},\tilde{\mathcal{F}},\tilde{P})$,
with the driving noises $\hat{W}$, $\hat{B}$ and the initial condition
$x_{0}$ on the time interval $[0,\hat{T}_{R})$ with

\[
\hat{T}_{R}:=\inf\big\{ t\geq0,\| \hat{B}\| _{1-\alpha,\infty,t}\geq R\big\} \wedge T,\quad R>0.
\]

Again, as above, we have a.s.~$\hat{T}_{R}=T$ for all $R$
large enough. So that $\hat{X},\hat{Y}$ satisfy the same SDE~\eqref{eq:1-1}, on $[0,T]$. Then by pathwise
uniqueness, we conclude that $\hat{X}(t)=\tilde{Y}(t)$ for all $t\in[0,T]$
$\tilde{P}$ a.s.. Hence, by applying Lemma \ref{lem:Gyo Kry}
in the Appendix we obtain the convergence of Euler's approximations
$X^{n}(t)$ to a process $X(t)$ in probability, uniformly in $t$
in $[0,T]$. Therefore, $\left\{ X(t),\, t\in[0,T]\right\} $ satisfy Equation (\ref{eq:1-1}).
\end{proof}
As a consequence we obtain the following existence result.
\begin{thm}
Assume that $b$, $\sigma_{W}$ and  $\sigma_{H}$ satisfy the hypotheses $\mathbf{(H.1)}-\mathbf{(H.2)}$.
If $1-H<\alpha<\min\left(\beta/2,1\right)$, then the Equation (\ref{eq:1-1}) has a unique strong solution.
\end{thm} 

\section*{Appendix}

In this appendix, we recall some results which play a great
role in this work.  We also show a technical lemma that have been used in the proof of pathwise uniqueness.
We begin with some a priori estimates from the paper of Nualart and Rascanu \cite{NR}.
\begin{prop}
\label{prop:Nua Ras 1}We have 
\begin{multline*}
(i)\,\|\int_{0}^{.}f(s)\, ds\|_{\alpha,t} \leq  C\int_{0}^{t}\dfrac{\left|f(s)\right|}{(t-s)^{\alpha}}\, ds.\\
(ii)\,\|\int_{0}^{.}f(s)\, dB^{H}(s)\|_{\alpha,t} \leq C\left\Vert B^{H}\right\Vert _{1-\alpha,\infty,t}\int_{0}^{t}\left((t-s)^{-2\alpha}+s^{-\alpha}\right)\|f\|_{\alpha,s}\, ds.
\end{multline*}
\end{prop}
Moreover, under the linear growth assumption,  we have from Nualart and Rascanu \cite{NR}, the following
\begin{prop}
\label{prop:Nua Ras 2}Assume $(H.1)$ and $(H.2)$. The following
estimates hold
\begin{multline*}
(j)\,\|\int_{0}^{.}b(s,f(s))\, ds\|_{\alpha,t} \leq C\left(\int_{0}^{t}\dfrac{\left|f(s)\right|}{(t-s)^{\alpha}}\, ds+1\right)\\
(jj)\,\|\int_{0}^{.}\sigma_{H}(s,f(s))\, dB^{H}(s)\|_{\alpha,t} \leq  C\left\Vert B^{H}\right\Vert _{1-\alpha,\infty,t}\int_{0}^{t}\left((t-s)^{-2\alpha}+s^{-\alpha}\right)\left(1+\|f\|_{\alpha,s}\right)\, ds
\end{multline*}
\end{prop}
We recall the following characterization of the convergence in probability
in term of weak convergence, see G\"{y}ongy and Krylov \textcolor{red}{\cite{GK}}. 
\begin{lem}
\label{lem:Gyo Kry}Let $(Z_{n})_{n\in\mathbb{N}}$ be a sequence
of random elements in a Polish space $(\mathcal{E},d)$ equipped with
the Borel $\sigma$-algebra. Then $(Z_{n})_{n\in\mathbb{N}}$ converges
in probability to an $\mathcal{E}$-valued random element if and only
if for every pair of subsequences $(Z_{m})_{m\in\mathbb{N}}$ and
$(Z_{k})_{k\in\mathbb{N}}$ there exists a subsequence $(Z_{m(p)},Z_{k(p)})_{p\in\mathbb{N}}$
converging weakly to a random element $v$ supported on the diagonal
$\left\{ (x,y)\in\mathcal{E}\times\mathcal{E}:x=y\right\} $.
\end{lem}
Finally, let us give a version of the Bihari's lemma.
\begin{lem}
\label{lem:Bihari}Let $1/2<\alpha<1$ and $c\geq0$ be fixed and
$f:\left[0,\infty\right)\longrightarrow\left[0,\infty\right)$ be
a continuous function such that 
\[
f(t)\leq a+bt^{\alpha}\int_{0}^{t}(t-s)^{-\alpha}s^{-\alpha}\varrho\left(f(s)\right)\, ds.
\]
where
$\varrho$ is a concave increasing function from $\mathbb{R}_{+}$
to $\mathbb{R}_{+}$ such that $\varrho(0)=0$, $\varrho(u)>0$ for
$u>0$ and satisfying \eqref{eq:rho est} for some $q>1$. Then for any $1<p<2$ such that $\alpha<1/p$ and $q>1$ with $1/p+1/q=1$
we have 
\[
f(t)\leq\left[F^{-1}\left(F(2^{q-1}a^{q})+2^{q-1}b^{q}\, C_{\alpha,p}^{q/p}t^{q\left((1/p)-\alpha\right)+1}\right)\right]^{1/q},
\]
for all $t\in\left[0,T\right]$ such that
\[
F(2^{q-1}a^{q})+2^{q-1}b^{q}\, C_{\alpha,p}^{q/p}t^{q\left((1/p)-\alpha\right)+1}\in Dom(F^{-1}),
\]
where 
\[
F(x)=\int_{1}^{x}\dfrac{du}{\varrho^{q}(u^{1/q})},\quad for\, x\geq0,
\]
and $F^{-1}$ is the inverse function of $F$. In particular, if moreover,
$a=0$ then $f(t)=0$ for all $0<t<T$.\end{lem}
\begin{proof}
Let $1<p<2$ such that $\alpha<1/p$. Using the H{\"o}lder inequality
we obtain
\[
f(t)\leq a+bt^{\alpha}\left(\int_{0}^{t}(t-s)^{-p\alpha}s^{-p\alpha}\, ds\right)^{1/p}\left(\int_{0}^{t}\varrho^{q}\left(f(s)\right)\, ds\right)^{1/q}
\]
For the first integral, using $s=tu$, we have the estimate
\[
\int_{0}^{t}(t-s)^{-p\alpha}s^{-p\alpha}\, ds=t^{1-2p\alpha}\int_{0}^{1}(1-u)^{-p\alpha}u^{-p\alpha}\, du=C_{\alpha,p}t^{1-2p\alpha}
\]
where $C_{\alpha,p}=B\left(1-p\alpha,1-p\alpha\right)$ is the beta function. It follows
that 
\[
f(t)\leq a+b\, C_{\alpha,p}^{1/p}t^{(1/p)-\alpha}\left(\int_{0}^{t}\varrho^{q}\left(f(s)\right)\, ds\right)^{1/q}.
\]
This yields
\[
f^{q}(t)\leq2^{q-1}a^{q}+2^{q-1}b^{q}\, C_{\alpha,p}^{q/p}t^{q\left((1/p)-\alpha\right)}\int_{0}^{t}\varrho^{q}\left(f(s)\right)\, ds.
\]
Then it follows from Bihari's Lemma, see \cite{Bihari56}, that 
\[
f(t)\leq\left[F^{-1}\left(F(2^{q-1}a^{q})+2^{q-1}b^{q}\, C_{\alpha,p}^{q/p}t^{q\left((1/p)-\alpha\right)+1}\right)
\right]^{1/q},
\]
for all such $t\in\left[0,T\right]$ such that
\[
F(2^{q-1}a^{q})+2^{q-1}b^{q}\, C_{\alpha,p}^{q/p}t^{q\left((1/p)-\alpha\right)+1}\in \mathrm{Dom}(F^{-1}).
\]
Now, it is simple to see from (\ref{eq:rho est}) that if $a=0$ then
$f(t)=0$ for $t\in[0,T]$.
\end{proof}

\end{document}